\documentclass[11pt,reqno]{amsart}

\usepackage{amsmath,amssymb,amsfonts,amsthm,latexsym,enumerate,tikz}
\usepackage[
  colorlinks=true,
  linkcolor=red,
  citecolor=red,
  urlcolor=red,
  pagebackref=true
]{hyperref}

\newcommand{\aut}{\mathrm{Aut}}

\newcommand{\id}{\mathrm{id}}
\newcommand{\sym}{\mathrm{Sym}}

\newcommand{\TFA}{\mathrm{TFA}}

\begin{document}

\newtheorem{theorem}{Theorem}[section]
\newtheorem{lemma}[theorem]{Lemma}
\newtheorem{corollary}[theorem]{Corollary}
\newtheorem{proposition}[theorem]{Proposition}
\newtheorem{conjecture}[theorem]{Conjecture}
\theoremstyle{definition}
\newtheorem{definition}[theorem]{Definition}
\newtheorem{problem}[theorem]{Problem}
\newtheorem{claim}[theorem]{Claim}
\newtheorem{question}[theorem]{Question}
\newtheorem{remark}[theorem]{Remark}
\newtheorem{answer}[theorem]{Answer}

\def\Answer{\par\noindent{\em Answer to Question}}

\date{}

\title[Stability of nontrivial graph pairs]{Stability of nontrivial graph pairs}

\author{Xiaomeng Wang}
\address{School of Mathematics and Statistics, Lanzhou University, Lanzhou, Gansu 730000, P.\,R. China}
\email{wangxiaomeng@lzu.edu.cn}

\author{Xing Gao}
\address{School of Mathematics and Statistics, Lanzhou University
Lanzhou, 730000, P.\,R. China; Gansu Provincial Research Center for Basic Disciplines of Mathematics
and Statistics, Lanzhou, 730070, P.\,R. China}
\email{gaoxing@lzu.edu.cn}


\begin{abstract}
A graph pair $(\Gamma, \Sigma)$ is called stable if every automorphism of the direct product $\Gamma\times\Sigma$ is induced componentwise by automorphisms of $\Gamma$ and $\Sigma$. A graph is twin-free if no two distinct vertices share the same neighbourhood in the graph. Two graphs $\Gamma$ and $\Sigma$ are coprime with respect to the direct product if there is no graph $\Delta$ of order greater than $1$ such that $\Gamma\cong\Gamma'\times\Delta$ and $\Sigma\cong\Sigma'\times\Delta$ for some graphs $\Gamma'$ and $\Sigma'$.
A graph pair $(\Gamma,\Sigma)$ is nontrivial if $\Gamma$ and $\Sigma$ are coprime connected twin-free graphs and exactly one of them is bipartite. In this paper, we prove that if $\Gamma$ is non-bipartite, stable, and factor-loopless, then each nontrivial graph pair $(\Gamma,\Sigma)$ is stable. This gives a partial answer to [Question~19, Qin, Xia and Zhou, Discrete Math., 113856, (2024)] and proves the factor-loopless case of [Conjecture~1.3, Wang, Qin and Xia, arXiv:2509.26170]. We also give affirmative answers to [Questions~3.5, 3.6, Gan, Liu and Xia, J. Combin. Theory Ser. B, 140--164, (2025)] and a negative answer to [Question~3.7, Gan, Liu and Xia, J. Combin. Theory Ser. B, 140--164, (2025)].

\medskip

\medskip

\end{abstract}

\maketitle

\section{Introduction}\label{sec1}

\subsection{Motivation}
Unless explicitly stated otherwise, all graphs considered in the main results are finite, undirected, loopless, and without multiple edges.
We follow the terminology and notation of
\cite{Big01} for graphs and \cite{Hun00} for groups.
For a graph $\Gamma$, we denote by $V(\Gamma)$, $A(\Gamma)$ and $E(\Gamma)$ its vertex set, arc set and edge set, respectively.
The neighbourhood of a vertex $u$ in $\Gamma$ is denoted by $N_\Gamma(u)$, or simply $N(u)$ when no confusion arises.
Let $\Gamma$ be a graph. The arc $(u, v)$ between $u$ and $v$ is an ordered pair of adjacent vertices.
The edge between two adjacent vertices $u, v$ of $\Gamma$ is denoted by $\{u, v\}$.
For any integer $n\geq 2$, $K_n$, $C_n$, $P_n$ stand for the complete graph, cycle and path on $n$ vertices, respectively.
An {\em automorphism} of a graph $\Gamma$ is a permutation $\sigma$ of $V(\Gamma)$ such that for any vertices $u, v$ of $\Gamma$, $\{u,v\}$ is an edge of $\Gamma$ if and only if $\{u^\sigma,v^\sigma\}$ is an edge of $\Gamma$. The set of all automorphisms of $\Gamma$, under composition, forms a group, called the \emph{automorphism group} of $\Gamma$ and is denoted by $\aut(\Gamma)$.

The {\em direct product} \cite{HIK2011} of two graphs $\Gamma$ and $\Sigma$, denoted by $\Gamma \times \Sigma$, is the graph with vertex set $V(\Gamma) \times V(\Sigma)$ and edge set $\{\{(u,v),(x,y)\}:\{u,x\}\in E(\Gamma), \{v,y\}\in E(\Sigma)\}$.
There is always a natural embedding
\begin{equation}\label{eq:direct}
\aut(\Gamma)\times\aut(\Sigma)\leq\aut(\Gamma\times\Sigma),
\end{equation}
where the product $\aut(\Gamma)\times\aut(\Sigma)$ acts componentwise. An automorphism $\tau\in\aut(\Gamma\times\Sigma)$ is called
\emph{unexpected} if $\tau\notin \aut(\Gamma)\times\aut(\Sigma)$. A pair $(\Gamma,\Sigma)$ is called \emph{stable} if
equality holds in \eqref{eq:direct}; otherwise it is called {\em unstable}.

The stability of a single graph, introduced in~\cite{MSZ1989}, is recovered from this definition by taking the second factor to be $K_2$. Indeed, a graph $\Gamma$ is stable in the classical sense if and only if the pair $(\Gamma, K_2)$ is stable.
The stability of graphs and graph pairs has attracted considerable attention; see, for example,~\cite{FH2022,HMM2021,LMS2015,MSZ1989,MSZ1992,Morris2021,NS1996,QXZ2019,QXZ2021,Surowski2001,Surowski2003,Wilson2008} for graphs and~\cite{GLX2025,QXZ2024,QXZZ2021,WQX2025,WXZ2025} for graph pairs.

A graph is called \emph{twin-free} (or, \emph{$R$-thin} \cite{WXZ2025}, \emph{vertex-determining} \cite{QXZ2019}) if no two distinct vertices have the same neighbourhood in the graph. Two graphs $\Gamma$ and $\Sigma$ are \emph{coprime} with respect to the direct product if there is no graph $\Delta$ of order greater than $1$ such that $\Gamma\cong\Gamma'\times\Delta$ and $\Sigma\cong\Sigma'\times\Delta$ for some graphs $\Gamma'$ and $\Sigma'$.

A classical result on direct products states that every pair of coprime, connected, non-bipartite, twin-free graphs is stable; see \cite[Theorem~8.18]{HIK2011}.
Together with \cite[Lemmas~3.3, 3.5]{QXZZ2021}, this reduces the central investigation of the stability problem to graph pairs in which exactly one factor is bipartite.

A graph pair $(\Gamma,\Sigma)$ is \emph{nontrivial} if $\Gamma$ and $\Sigma$ are coprime connected twin-free graphs and exactly one of them is bipartite. Moreover, a nontrivial graph pair is called \emph{nontrivially unstable} if it is unstable.

The following question, due to Qin, Xia and Zhou, is a central motivation for this paper.

\begin{question}[{\cite[Question~19]{QXZ2024}}]\label{que}
\label{que:queqxzz}
Let $(\Gamma,\Sigma)$ be a nontrivial graph pair with $\Sigma$ bipartite. Under what conditions does it hold that $(\Gamma, \Sigma)$ is stable if and only if $\Gamma$ is stable?
\end{question}

This question asks when the stability of the graph pair can be completely reduced to the stability of the non-bipartite factor. Even when one imposes strong restrictions on the two factors, this reduction is far from automatic. Along this line, Wang, Qin and Xia proposed the following conjecture.

\begin{conjecture}\cite[Conjecture~1.3]{WQX2025}\label{conj:conjsps}
Let $(\Gamma,\Sigma)$ be a nontrivial graph pair without loops, where $\Sigma$ is bipartite. Then $(\Gamma,\Sigma)$ is stable if and only if $\Gamma$ is stable.
\end{conjecture}

The present paper proves this conjecture under the additional assumption that the non-bipartite factor is factor-loopless.

\subsection{Main results}
Let
\[\Gamma\cong\Gamma_1\times\cdots\times\Gamma_n
\]
 be the direct-prime factorization of a connected non-bipartite twin-free graph in the loops-allowed category for direct products. We call $\Gamma$ factor-loopless if each direct-prime factor $\Gamma_i$ is loopless. For an integer $q\geq 1$, let $\Gamma^q:=\Gamma\times\cdots\times\Gamma$ denote the direct product of $q$ copies of $\Gamma$. Our main result is the following theorem.

\begin{theorem}
\label{th:thgpsg}
Let $(\Gamma,\Sigma)$ be a nontrivial graph pair, where $\Gamma$ is factor-loopless and $\Sigma$ is bipartite. Then $(\Gamma,\Sigma)$ is stable if and only if $\Gamma$ is stable.
\end{theorem}

 Theorem~\ref{th:thgpsg} gives a partial answer to Question~\ref{que:queqxzz}, and proves the factor-loopless case of Conjecture~\ref{conj:conjsps}.
 The ``only if'' part of Theorem \ref{th:thgpsg} follows from \cite[Theorem~1.4(a)]{WQX2025}.
 The main content of
 Theorem \ref{th:thgpsg} is the converse implication: if the non-bipartite factor $\Gamma$ is stable and factor-loopless, then no bipartite coprime twin-free graph $\Sigma$ can produce unexpected automorphisms.
To prove Theorem \ref{th:thgpsg}, we introduce the following definition.

\begin{definition}
\label{de:desymp}
A connected non-bipartite stable graph $\Gamma$ is called \emph{Sym-prime} if every nontrivial graph pair $(\Gamma,\Sigma)$ with $\Sigma$ connected and bipartite is stable.
\end{definition}

 By Definition \ref{de:desymp}, a Sym-prime graph cannot become part of an unstable graph pair by taking the direct product with a connected coprime bipartite twin-free factor.
By \cite[Theorem~1.7]{GLX2025} and \cite[Theorem~1.5]{WQX2025}, $K_n$ ($n\geq 3$) and $C_{2k+1}$ ($k\geq 1$) are Sym-prime.

The first key step is the prime power case.

\begin{theorem}
\label{th:thpsp}
Let $\Gamma$ be a connected, non-bipartite twin-free, and loopless graph which is prime with respect to the direct product. If $\Gamma^q$ is stable for some integer $q\geq 1$, then $\Gamma^q$ is Sym-prime.
\end{theorem}

The second key step passes from prime powers to arbitrary factor-loopless graphs.

\begin{theorem}
\label{th:thfless}
Let $\Gamma$ be a connected non-bipartite factor-loopless graph with respect to the direct product. If $\Gamma$ is stable, then $\Gamma$ is Sym-prime.
\end{theorem}

Theorem \ref{th:thgpsg} follows immediately from Theorem \ref{th:thfless}. Indeed, if $\Gamma$ is factor-loopless and stable, then $\Gamma$ is Sym-prime by Theorem \ref{th:thfless}, hence every nontrivial graph pair $(\Gamma,\Sigma)$ with $\Sigma$ bipartite is stable.

\subsection{Organization of the paper}
The paper is organized as follows. In Section \ref{sec:secp}, we recall the necessary preliminaries on two-fold automorphisms, Cartesian products and Cartesian skeletons, and unexpected automorphisms.
We will prove Theorem~\ref{th:thpsp} in Section~\ref{sec:secpsp}. In Section~\ref{sec:secthppsp}, we first prove Theorem~\ref{th:thfless}, which then yields
Theorem~\ref{th:thgpsg}.
In Section \ref{sec:seca}, we answer the questions proposed by Gan, Liu and Xia in \cite{GLX2025}.

\section{Preliminaries}
\label{sec:secp}

In this section, we recall and introduce some definitions and results concerning TF-automorphisms, Cartesian products of graphs and unexpected automorphisms. This notation and these results will be used in the proof of the main result of this paper.

\subsection{TF-automorphisms}
Let $\Gamma$ and $\Sigma$ be graphs, and let $\alpha$ and $\beta$ be bijections from $V(\Gamma)$ to $V(\Sigma)$. The pair $(\alpha, \beta)$ is called a {\em two-fold isomorphism} \cite{LM2011} (or, {\em TF-isomorphism}) from $\Gamma$ to $\Sigma$ if $(u, v)\in A(\Gamma)$ if and only if $(u^\alpha,v^\beta)\in A(\Sigma)$.
In particular, if $\Sigma=\Gamma$, then we call $(\alpha,\beta)$ a {\em two-fold automorphism} (or, {\em TF-automorphism}) of $\Gamma$.
The set of TF-automorphisms of $\Gamma$ forms a group under the product
\[
(\alpha_1,\beta_1)(\alpha_2,\beta_2):=(\alpha_1\alpha_2,\beta_1\beta_2),
\]
and we denote this group by $\TFA(\Gamma)$. A TF-automorphism $(\alpha,\beta)$ is \emph{diagonal} if $\alpha=\beta$.

A diagonal TF-automorphism $(\alpha,\alpha)$ corresponds uniquely to an automorphism $\alpha$, which gives $\aut(\Gamma)\leq\TFA(\Gamma)$ for each graph $\Gamma$. Moreover, the equality holds if and only if every TF-automorphism of $\Gamma$ is diagonal.
Lauri, Mizzi and Scapellato \cite{LMS2015} gave a useful criterion for a connected non-bipartite graph to be unstable. We restate their result as follows.

\begin{lemma}~\cite[Theorem~3.2]{LMS2015}\label{lem:TF}
Let $\Gamma$ be a connected and non-bipartite graph. Then $\aut(\Gamma\times K_2)=\TFA(\Gamma)\rtimes \mathbb{Z}_2$. In particular, $\Gamma$ is unstable if and only if it admits a non-diagonal TF-automorphism.
\end{lemma}

\subsection{Cartesian product of graphs}

The \emph{Cartesian product} \cite{HIK2011} $\Gamma\Box\Sigma$ of graphs $\Gamma$ and $\Sigma$ is the graph with vertex set $V(\Gamma)\times V(\Sigma)$ in which $(u,i)$ is adjacent to $(v,j)$ if and only if either $u=v$ and $i$ is adjacent to $j$ in $\Sigma$, or $i=j$ and $u$ is adjacent to $v$ in $\Gamma$.
An expression $$\Gamma\cong\Gamma_1\Box\cdots\Box\Gamma_n$$ with each $\Gamma_i$ Cartesian-prime is called a \emph{Cartesian-prime factorization} of $\Gamma$.
Similar to the direct product of graphs, two graphs $\Gamma$ and $\Sigma$ are \emph{Cartesian-coprime} if there is no graph $\Delta$ of order greater than $1$ such that $\Gamma\cong\Gamma'\Box\Delta$ and $\Sigma\cong\Sigma'\Box\Delta$ for some graphs $\Gamma'$ and $\Sigma'$. They are \emph{Cartesian-quasicoprime} if for every common Cartesian-prime factor $\Delta$ of $\Gamma$ and $\Sigma$, the Cartesian-prime factorization of neither $\Gamma$ nor $\Sigma$ contains at least two factors isomorphic to $\Delta$.

We shall use the following standard facts. They follow from ~\cite[Theorems~6.6,~6.13,~6.21]{HIK2011} together with~\cite[Corollary~6.11]{HIK2011}.

\begin{lemma}\label{thm:cartesian}
The following statements hold:
\begin{enumerate}[\rm(a)]
\item\label{thm:cartesianc} If $\Gamma$ and $\Sigma$ are both connected graphs, then $\Gamma\Box\Sigma$ is connected.
\item\label{thm:cartesiand} If $\Gamma$, $\Sigma$ and $\Delta$ are graphs such that $\Gamma\Box\Delta\cong\Sigma\Box\Delta$ and $\Delta$ is nonempty, then $\Gamma\cong\Sigma$.
\item\label{thm:cartesiana} Every connected graph has a unique Cartesian-prime factorization up to isomorphism and the order of the factors.
\item\label{thm:cartesianb} The automorphism group of a Cartesian product of connected Cartesian-prime graphs is the direct product of the wreath products on the sets of pairwise isomorphic Cartesian-prime graphs.
\end{enumerate}
\end{lemma}

We now recall the definition of Cartesian skeleton of a graph.

\begin{definition}\label{def:skeleton}
For a graph $\Gamma$, the {\em Boolean square} of $\Gamma$, denoted by $B(\Gamma)$, is the graph with vertex set $V(\Gamma)$ and edge set
\[
\{\{u,v\}:u, v\in V(\Gamma),\,u\neq v,\,\,N_{\Gamma}(u)\cap N_{\Gamma}(v)\neq\varnothing\}.
\]
An edge $\{u,v\}$ of $B(\Gamma)$ is said to be \emph{dispensable} with respect to $\Gamma$ if there exists some $w\in V(\Gamma)$ satisfying both of the conditions:
\begin{itemize}
\item $N_\Gamma(u)\cap N_\Gamma(v)\subsetneq N_\Gamma(u)\cap N_\Gamma(w)$ or $N_\Gamma(u)\subsetneq N_\Gamma(w)\subsetneq N_\Gamma (v)$, and
\item $N_\Gamma(v)\cap N_\Gamma(u)\subsetneq N_\Gamma(v)\cap N_\Gamma(w)$ or $N_\Gamma(v)\subsetneq N_\Gamma(w)\subsetneq N_\Gamma (u)$.
\end{itemize}
The {\em Cartesian skeleton} $S(\Gamma)$ of $\Gamma$ is the graph obtained from $B(\Gamma)$ by removing all dispensable edges with respect to $\Gamma$.
\end{definition}

The following basic facts about Cartesian skeletons will be needed.

\begin{lemma}[{\cite[Propositions~8.10~and~8.13]{HIK2011}}]
\label{le:lecsb}
Let $\Gamma$ and $\Sigma$ be graphs. Then the following statements hold:
\begin{enumerate}[\rm(a)]
\item\label{le:leas} Each automorphism of $\Gamma$ is also an automorphism of $S(\Gamma)$.
\item\label{it:itcsba} If $\Gamma$ is connected and non-bipartite, then $S(\Gamma)$ is connected.
\item\label{it:itcsbb} If $\Gamma$ is connected and bipartite, then $S(\Gamma)$ has precisely two connected components, whose respective vertex sets are the two partite sets of $\Gamma$.
\item\label{le:lecsd} If $\Gamma$ and $\Sigma$ are both twin-free without isolated vertices, then $S(\Gamma\times \Sigma)=S(\Gamma)\Box S(\Sigma)$.
\end{enumerate}
\end{lemma}

\subsection{Unexpected automorphisms} Let $\Gamma$ and $\Sigma$ be graphs. The $\Gamma$-partition and $\Sigma$-partition of $V(\Gamma\times\Sigma)$ are the partitions $\{\{u\}\times V(\Sigma): u\in V(\Gamma)\}$ and $\{V(\Gamma)\times \{x\}: x\in V(\Sigma)\}$, respectively.

Gan, Liu and Xia in \cite{GLX2025} introduced the following definition to describe the action of unexpected automorphisms.

\begin{definition}\cite[Definition~2.2]{GLX2025}
\label{de:degmsm}
Let $\Gamma$ and $\Sigma$ be graphs. An automorphism of $\Gamma\times\Sigma$ is called a {\em $\Gamma$-mixer} ({\em $\Sigma$-mixer}) if it breaks the $\Gamma$-partition ($\Sigma$-partition) of $V(\Gamma\times\Sigma)$.
\end{definition}

Wang, Qin and Xia in~\cite{WQX2025} proved that for any nontrivial graph pair $(\Gamma,\Sigma)$, where $\Gamma$ is non-bipartite and $\Sigma$ is bipartite with bipartition $\{U,W\}$, there exists an automorphism of $\Gamma\times\Sigma$, which stabilizes $V(\Gamma)\times U$.

\begin{lemma}\cite[Lemma~2.4]{WQX2025}
\label{claim1}
Let $\Gamma$ and $\Sigma$ be connected twin-free graphs such that $\Gamma$ is non-bipartite and $\Sigma$ is bipartite with bipartition $\{U,W\}$, and let $\tau$ be an automorphism of $S(\Gamma)\Box S(\Sigma)$. Then there exists $\sigma\in\aut(S(\Sigma))$ with $|\sigma|\leq2$ such that $(\mathrm{id},\sigma)\tau$ and $\tau(\mathrm{id},\sigma)$ stabilize $V(\Gamma)\times U$, where $\id$ is the identity permutation on $V(\Gamma)$, and if $|\sigma|=2$, then $\sigma$ interchanges $U$ and $W$.
\end{lemma}

Using Lemma \ref{claim1}, the authors of \cite{WQX2025} proved that for any nontrivial graph pair $(\Gamma,\Sigma)$ with $\Gamma$ non-bipartite and $\Sigma$ bipartite, a $\Sigma$-mixer of $\Gamma\times\Sigma$ is also a $\Gamma$-mixer.

\begin{lemma}\cite[Lemma~4.4]{WQX2025}
\label{le:lestg}
Let $\Gamma$ and $\Sigma$ be connected graphs such that $\Gamma$ is non-bipartite and $\Sigma$ is bipartite with bipartition $\{U,W\}$, and let $\rho$ be an automorphism of $S(\Gamma)\Box S(\Sigma)$ that stabilizes $V(\Gamma)\times U$ and is a $\Sigma$-mixer of $V(\Gamma)\times V(\Sigma)$. Then $\rho$ is a $\Gamma$-mixer of $V(\Gamma)\times V(\Sigma)$.
\end{lemma}

In fact, if a nontrivial graph pair $(\Gamma,\Sigma)$ admits unexpected automorphisms and $\Gamma$ is stable, then $\Gamma\times\Sigma$ admits a $\Sigma$-mixer.

\begin{lemma}\label{le:lepn}
Let $\Gamma$ and $\Sigma$ be connected graphs such that $\Gamma$ is non-bipartite and $\Sigma$ is bipartite with bipartition $\{U,W\}$.
Assume that $\Gamma$ is stable and that $\Gamma\times\Sigma$ admits an unexpected automorphism. Then there exists a $\Sigma$-mixer $\delta\in \aut(\Gamma\times\Sigma)$.
\end{lemma}

\begin{proof}
 Suppose, to the contrary, that no automorphism of $\Gamma\times\Sigma$ is a $\Sigma$-mixer. Then all automorphisms of $\Gamma\times\Sigma$ are $\Sigma$-preserving. By \cite[Lemma~2.6(b)]{QXZZ2021}, $(\Gamma,\Sigma)$
is unstable if and only if there is a non-diagonal $\Sigma$-automorphism $(\alpha_1,\alpha_2,\ldots,\alpha_m)$ of $\Gamma$, where $m=|V(\Sigma)|$.

We now show that this is impossible. Let $\{i,j\}\in E(\Sigma)$. By definition, for all $x,y\in V(\Gamma)$, $(x,y)\in A(\Gamma)$ if and only if $(x^{\alpha_i},y^{\alpha_j})\in A(\Gamma)$.
Thus $(\alpha_i,\alpha_j)$ is also a two-fold automorphism of $\Gamma$. Since $\Gamma$ is connected, non-bipartite, and stable, Lemma \ref{lem:TF} implies that every two-fold automorphism of $\Gamma$ is diagonal. Therefore
$\alpha_i=\alpha_j$
whenever $\{i,j\}\in E(\Sigma)$. Since $\Sigma$ is connected, we obtain $$\alpha_1=\alpha_2=\cdots=\alpha_m.$$
Hence every $\Sigma$-automorphism of $\Gamma$ is diagonal, a contradiction. Therefore some unexpected automorphism of $\Gamma\times\Sigma$ is a $\Sigma$-mixer.
\end{proof}

Next, we introduce some notations concerning unexpected automorphisms of nontrivial graph pairs.

Let $(\Gamma,\Sigma)$ be a nontrivial graph pair, where $\Gamma$ is non-bipartite and $\Sigma$ is bipartite with bipartition $\{U,W\}$. By Lemma \ref{le:lecsb}\eqref{it:itcsbb}, $S(\Sigma)$ has two connected components.
Denote by $S^{+}(\Sigma)$ and $S^{-}(\Sigma)$ the two connected components of $S(\Sigma)$ whose vertex sets are $U$ and $W$, respectively. Set $$X:=V(\Gamma)\times U, \qquad Y:=V(\Gamma)\times W.$$
We assume that $\delta$ is an unexpected automorphism of $\Gamma\times\Sigma$ as shown in Lemma \ref{le:lepn} when $\Gamma$ is stable.
Then $\delta$ is a $\Sigma$-mixer of $\Gamma\times\Sigma$.
 By Lemma~\ref{le:lecsb}\eqref{le:leas}, the automorphism $\delta$ of $\Gamma\times\Sigma$ is also an automorphism of $S(\Gamma\times\Sigma)$.
 Moreover, by Lemma \ref{le:lecsb}\eqref{le:lecsd}, $$S(\Gamma\times\Sigma)=S(\Gamma)\Box S(\Sigma).$$
 By Lemma \ref{claim1}
there exists $\sigma\in\aut(S(\Sigma))$ with $|\sigma|\le 2$ such that $\rho:=(\operatorname{id},\sigma)\delta$
stabilizes $X$.
 Combining this with Lemma~\ref{le:lecsb}(c), the automorphism $\rho$ also stabilizes $Y$. We call $\rho$ the {\em $S$-automorphism} induced by $\delta$.
Set $$\rho_{+}:=\rho|_{X}\in \aut(S(\Gamma)\Box S^{+}(\Sigma)), \qquad  \rho_{-}:=\rho|_{Y}\in \aut(S(\Gamma)\Box S^{-}(\Sigma)).$$ We say that $\rho$ is \emph{one-sided} if exactly one of $\rho_{+}$ and $\rho_{-}$ is a $\Sigma$-mixer and $\rho$ is \emph{balanced} if both $\rho_{+}$ and $\rho_{-}$ are $\Sigma$-mixers.
We now introduce some notation concerning Cartesian decompositions of $\Gamma$ and $\Sigma$. Using Lemma \ref{thm:cartesian} \eqref{thm:cartesianc}, \eqref{thm:cartesiana}, after reordering Cartesian-prime factors of $S(\Gamma)$ and $S(\Sigma)$, as shown in \cite{GLX2025} there are Cartesian-product decompositions
\begin{align*}
S(\Gamma)=D_0\Box D_1\Box D_2\Box D_3,
\qquad
S^{+}(\Sigma)= A_0\Box A_1\Box A,
\qquad
S^{-}(\Sigma)= B_0\Box B_2\Box B,
\end{align*}
such that the following hold:
\begin{enumerate}
\item $D_0$ is the product of the Cartesian-prime factors of $S(\Gamma)$ exchanged with factors of both $S^{+}(\Sigma)$ and $S^{-}(\Sigma)$;
\item $D_1$ is the product of the Cartesian-prime factors of $S(\Gamma)$ exchanged with factors of $S^{+}(\Sigma)$ but not with factors of $S^{-}(\Sigma)$;
\item $D_2$ is the product of the Cartesian-prime factors of $S(\Gamma)$ exchanged with factors of $S^{-}(\Sigma)$ but not with factors of $S^{+}(\Sigma)$;
\item $D_3$ is the product of the Cartesian-prime factors of $S(\Gamma)$ not exchanged with either side.
\end{enumerate}

Gan, Liu and Xia in \cite[Lemma~4.1]{GLX2025} proved that if $D_0$ is nontrivial, then $\Gamma$ and $\Sigma$ have a common factor with respect to the direct product. Since we only consider nontrivial graph pairs, $\Gamma$ and $\Sigma$ are coprime. Therefore we assume that $D_0$ is trivial in this paper. That is,
\[
S(\Gamma)=D_1\Box D_2\Box D_3\qquad  S^{+}(\Sigma)=A_1\Box A\qquad  S^{-}(\Sigma)=B_2\Box B.
\]
With respect to these decompositions, the actions of $\rho_{+}$ and $\rho_{-}$ have the forms
\[
\bigl((d_1,d_2,d_3),(a_1,a)\bigr)^{\rho_{+}}
=
\bigl((a_1,d_2,d_3)^{\eta},(d_1,a)^{\alpha}\bigr)
\]
and
\[
\bigl((d_1,d_2,d_3),(b_2,b)\bigr)^{\rho_{-}}
=
\bigl((d_1,b_2,d_3)^{\theta},(d_2,b)^{\beta}\bigr),
\]
where
\begin{align*}
\eta:& A_1\Box D_2\Box D_3\to S(\Gamma),\\
\alpha:& D_1\Box A\to S^{+}(\Sigma),\\
\theta:& D_1\Box B_2\Box D_3\to S(\Gamma),\\
\beta:& D_2\Box B\to S^{-}(\Sigma)
\end{align*}
are graph isomorphisms.

\section{The proof of Theorem \ref{th:thpsp}}\label{sec:secpsp}

Let $\Gamma$ be a non-bipartite loopless prime graph. Set $\Delta:=\Gamma^q$ for some $q\geq 1$. If $\Delta$ is not Sym-prime, then there exists a bipartite graph $\Sigma$ such that the nontrivial graph pair $(\Delta,\Sigma)$ is unstable. Let $\delta$ be an unexpected automorphism of $\Delta\times\Sigma$ and $\rho$ be the $S$-automorphism induced by $\delta$. Then $\rho$ is either one-sided or balanced. We will therefore consider these two cases separately to prove Theorem \ref{th:thpsp}.

\subsection{One-sided}
We first handle the case that $\rho$ is one-sided.
Without loss of generality, suppose that $\rho_{+}$ is a $U$-mixer and $\rho_{-}$ is side-preserving. The other case can be treated similarly. The following result is well known in the literature. We give its proof for  completeness.

\begin{lemma}\label{le:leosd}
Let $(\Gamma, \Sigma)$ be a nontrivial graph pair, where $\Sigma$ is bipartite with bipartition $\{U,W\}$, and let $\delta$ be an unexpected automorphism of $\Gamma\times\Sigma$ and $\rho$ be the $S$-automorphism induced by $\delta$ such that $\rho_+$ is $U$-mixer and $\rho_-$ is side-preserving.
Then, after reordering Cartesian-prime factors, the following statements hold:
\begin{enumerate}
\item\label{it:itleosd1} there exist decompositions $S(\Gamma)=M\Box R$ and $S^{+}(\Sigma)=N\Box T$, where $M$ and $N$ are nontrivial connected graphs with $M\cong N$, such that
\[((m,r),(n,t))^{\rho_{+}}=\bigl((n^\mu,r^\alpha),(m^\nu,t^\beta)\bigr)\]
for some isomorphisms $\mu:N\to M$ and $\nu:M\to N$, and automorphisms $\alpha\in\aut(R)$ and $\beta\in\aut(T)$.
\item\label{it:itleosd2} there exist $ \lambda\in\aut(S(\Gamma))$ and $\eta\in\aut(S^{-}(\Sigma))$ such that
$\rho_{-}=(\lambda,\eta)$.
\end{enumerate}
\end{lemma}

\begin{proof}
\eqref{it:itleosd1}. Since $\Gamma$ and $\Sigma$ are connected graphs, by Lemma \ref{le:lecsb}\eqref{it:itcsba} $S(\Gamma)$ is connected and $S(\Sigma)$ has two connected components $S^+(\Sigma)$ and $S^-(\Sigma)$. Lemma \ref{thm:cartesian}\eqref{thm:cartesiana} implies that $S(\Gamma)$, $S^+(\Sigma)$ and $S^-(\Sigma)$ each have a unique Cartesian-prime factorization. By Lemma \ref{thm:cartesian}\eqref{thm:cartesianb} again, the automorphism $\rho$ of $S(\Gamma)\Box S^{+}(\Sigma)$ permutes isomorphic Cartesian-prime factors and acts factorwise after such a permutation.

 Since $\rho_{+}$ is a $U$-mixer, $S(\Gamma)$ and $S^+(\Sigma)$ have a common factor (with respect to the Cartesian product). By Lemma \ref{thm:cartesian}(b) we can collect all common factors of $S(\Gamma)$ and $S^+(\Sigma)$ to obtain $R$ and $T$, respectively. Thus
 $$S(\Gamma)=M\Box R, \qquad S^{+}(\Sigma)=N\Box T,$$
where $M\cong N$. This proves part \eqref{it:itleosd1}.

\eqref{it:itleosd2} Since $\rho_{-}$ is side-preserving on the $W$-side,
we have $$\rho_{-}\in\aut(S(\Gamma))\times\aut(S^{-}(\Sigma)).$$
Thus $\rho_{-}=(\lambda,\eta)$ for some $\lambda\in\aut(S(\Gamma))$ and $\eta\in\aut(S^-(\Sigma))$.
This yields the result.
\end{proof}

\begin{lemma}\label{pr:prro}
Let $(\Gamma, \Sigma)$ be a nontrivial graph pair, where $\Sigma$ is bipartite with bipartition $\{U,W\}$, let $\delta$ be an unexpected automorphism of $\Gamma\times\Sigma$ and $\rho$ the $S$-automorphism induced by $\delta$.
If $\rho$ is one-sided, then $\Gamma$ admits a non-diagonal TF-automorphism.
\end{lemma}

\begin{proof}
Since $\rho$ is one-sided, by Lemma~\ref{le:leosd} we may assume $$S(\Gamma)=M\Box R,\qquad  S^{+}(\Sigma)=N\Box T$$
with $$((m,r),(n,t))^{\rho_{+}}=((n^\mu,r^\alpha),(m^\nu,t^\beta)),$$
and $\rho_{-}=(\lambda,\eta)$ for suitable isomorphisms and automorphisms as in Lemma \ref{le:leosd}.

Fix a vertex $t_0\in V(T)$. Since $\Sigma$ is connected and bipartite, every vertex of $U=V(N)\times V(T)$ has a neighbour in $W$. Therefore, for each bijection $f:V(M)\to V(N)$
and each $m\in V(M)$, we can choose a vertex $\omega_m^f\in W$
adjacent to $(f(m),t_0)\in U$.
Fix such a bijection $f$. Define $\xi_f\in\sym(V(\Gamma))$ by
$$(m,r)^{\xi_f}=((f(m))^\mu,r^\alpha)$$ for all $$(m,r)\in V(M)\times V(R)=V(\Gamma),$$
and define $\zeta=\lambda\in\sym(V(\Gamma))$.

Since $f$, $\mu$, and $\alpha$ are bijections, $\xi_f$ is a permutation of $V(\Gamma)$.
Let $x=(m,r)\in V(\Gamma)$ and $y\in V(\Gamma)$. Set
$$X_x^f:=\bigl(x,(f(m),t_0)\bigr)\in V(\Gamma)\times U$$
and
$$Y_{y,m}^f:=\bigl(y,\omega_m^f\bigr)\in V(\Gamma)\times W.$$
Note that the second coordinates of $X_x^f$ and $Y_{y,m}^f$ are adjacent in $\Sigma$ by the choice of $\omega_m^f$. So, if $x$ is adjacent to $y$ in $\Gamma$, then $X_x^f$ is adjacent to $Y_{y,m}^f$ in $\Gamma\times\Sigma$.
Since $\delta$ is an automorphism of $\Gamma\times\Sigma$,
$(x,y)\in A(\Gamma)$ implies that $$\{(X_x^f)^{\delta},(Y_{y,m}^f)^{\delta}\}\in E(\Gamma\times\Sigma).$$
Taking first coordinates and using $\pi_1\delta=\pi_1\rho$, if $(x,y)\in A(\Gamma)$, then
 $$((X_x^f)^{\rho\pi_1},(Y_{y,m}^f)^{\rho\pi_1})\in A(\Gamma).$$

Since $X_x^f\in X$, we have $$(X_x^f)^{\rho\pi_1}
=((x,(f(m),t_0))^{\rho_{+}})^{\pi_1}
=((f(m)^\mu),r^\alpha)
=x^{\xi_f}.$$
It follows from $Y_{y,m}^f\in Y$ and $\rho_{-}=(\lambda,\eta)$ that $$(Y_{y,m}^f)^{\rho\pi_1}
=((y,\omega_m^f)^{\rho_{-}})^{\pi_1}
=y^{\lambda}
=y^{\zeta}.$$
Therefore for all $x,y\in V(\Gamma)$, $(x^{\xi_f},y^{\zeta})\in A(\Gamma)$ if $(x,y)\in A(\Gamma)$.

Since $\xi_f$ and $\zeta$ are permutations of $V(\Gamma)$ and $\Gamma$ is finite, the map $\psi: A(\Gamma)\to A(\Gamma)$, defined by
$(x,y)\mapsto (x^{\xi_f},y^{\zeta})$,
is injective. Hence $$(x,y)\in A(\Gamma)\Longleftrightarrow
(x^{\xi_f},y^{\zeta})\in A(\Gamma)$$
for all $x,y\in V(\Gamma)$. Thus $(\xi_f,\zeta)$ is a two-fold automorphism of $\Gamma$.

It remains to prove that $(\xi_f,\zeta)$ is non-diagonal. Since $M$ is nontrivial, $|V(M)|\ge 2$. Hence there exist two distinct bijections
\[f_1,f_2:V(M)\to V(N).\]
The corresponding permutations $\xi_{f_1}$ and $\xi_{f_2}$ are distinct because $\mu$ is bijective. Since $\zeta$ is fixed, at least one of $\xi_{f_1}$ and $\xi_{f_2}$ is different from $\zeta$. Choosing such a bijection $f$, we obtain a non-diagonal two-fold automorphism $(\xi_f,\zeta)$ of $\Gamma$.
\end{proof}

\subsection{Balanced}
We now deal with the case where $\rho$ is balanced and show that it cannot occur when $\Gamma$ is factor-loopless.

\begin{lemma}\label{le:lelp}
Let $(\Gamma, \Sigma)$ be a nontrivial graph pair, where $\Gamma$ is non-bipartite and $\Sigma$ is bipartite. Assume that $\Gamma$ is factor-loopless. Then $\Gamma\times\Sigma$ admits no balanced $\Sigma$-mixer.
\end{lemma}

\begin{proof}
Suppose, for a contradiction, that there exists a $\Sigma$-mixer $\rho$.
As shown in Section \ref{sec:secp}, we may assume
\[
S(\Gamma)=D_1\Box D_2\Box D_3,
\qquad
S^{+}(\Sigma)= A_1\Box A,
\qquad
S^{-}(\Sigma)= B_2\Box B.
\]

If $\Gamma$ is prime, by \cite[Lemma~4.3]{GLX2025}, the factor $D_3$ is trivial. Otherwise, this would contradict the fact that $\Gamma$ is prime. We assume that $D_3$ is trivial, and that $D_1$ and $D_2$ are nontrivial. Then $|D_1|, |D_2|>1$.
Since $\Gamma$ is connected, $\Gamma$ has at least one edge. Suppose that $$\{(c_1,c_2),(d_1,d_2)\}$$ is an edge in $\Gamma$, where $c_1, d_1\in V(D_1)$ and $c_2, d_2\in V(D_2)$.
Since $\Gamma$ is undirected,
$$\{(d_1,d_2),(c_1,c_2)\}$$ is also an edge in $\Gamma$.
Apply \cite[Lemma~4.1]{GLX2025} to the two edges
$$\{(c_1,c_2),(d_1,d_2)\},\qquad \{(d_1,d_2),(c_1,c_2)\}.$$
By \cite[Lemma~4.1]{GLX2025} again, $\{(d_1,c_2),(d_1,c_2)\}$ is also an edge in $\Gamma$ .
Thus $\Gamma$ has a loop, a contradiction. So $\Gamma\times\Sigma$ admits no balanced $\Sigma$-mixer when $\Gamma$ is prime.

We now consider the case where $\Gamma$ is not prime with respect to the direct product.
Define \[
V(\Delta_{12})=V(D_1\Box D_2),
\qquad
V(\Delta_3)=V(D_3),
\]
where $
(c_1,c_2)$ is adjacent to $(d_1,d_2)$ in $\Delta_{12}$
if and only if there exist $c_3,d_3\in V(D_3)$ such that
$(c_1,c_2,c_3)$ is adjacent to $(d_1,d_2,d_3)$ in $\Gamma$, and $c_3$ is adjacent to $d_3$ in $\Delta_3$
if and only if there exist
$$(c_1,c_2),(d_1,d_2)\in V(D_1\Box D_2)$$ such that $
(c_1,c_2,c_3)$ is adjacent to $(d_1,d_2,d_3)$ in $\Gamma$.
By \cite[Lemma~4.3]{GLX2025}, we have
$\Gamma\cong \Delta_{12}\times\Delta_3$.

 Since
$\rho$ is balanced, $D_1$ and $D_2$
have order greater than $1$.
Since $\Gamma\cong \Delta_{12}\times\Delta_3$, the direct-product factor
$\Delta_{12}$ has at least one edge. Choose an edge $$\{(c_1,c_2),(d_1,d_2)\}\in\Delta_{12}.$$
By the definition of $\Delta_{12}$, there exist $c_3,d_3\in V(D_3)$ such that $(c_1,c_2,c_3)$ is adjacent to $(d_1,d_2,d_3)$ in $\Gamma$.
Since $\Gamma$ is undirected, $$\{(d_1,d_2,d_3),(c_1,c_2,c_3)\}$$ is also an edge in $\Gamma$.

Apply \cite[Lemma~4.1]{GLX2025} to the two edges $$\{(c_1,c_2,c_3),(d_1,d_2,d_3)\},\qquad \{(d_1,d_2,d_3),(c_1,c_2,c_3)\}.$$
We obtain that $(d_1,c_2,c_3)$ is adjacent to $(d_1,c_2,d_3)$ in $\Gamma$.
By the definition of $\Delta_{12}$, this means exactly that there is an edge $$\{(d_1,c_2),(d_1,c_2)\}\in E(\Delta_{12}).$$
Thus $\Delta_{12}$ has a loop. Since $\Gamma\cong\Delta_{12}\times\Delta_3$ and the factorization is unique, there exists a direct-product factor $\Delta_{12}$ of $\Gamma$ with loops. This is a contradiction. So this case cannot occur.

Thus $\Gamma\times\Sigma$ does not admit a balanced $\Sigma$-mixer.
\end{proof}

We now give the proof of Theorem \ref{th:thpsp}.

\begin{proof}[Proof of Theorem~\ref{th:thpsp}]
Set $\Delta=\Gamma^q$ for some $q\geq 1$.
Suppose, for a contradiction, that $\Delta$ is not Sym-prime. Then there exists a bipartite graph $\Sigma$ such that $(\Delta,\Sigma)$ is a nontrivially unstable graph pair.
Since $\Delta$ is stable, Lemma \ref{le:lepn} gives a $\Sigma$-mixer of $\Delta\times\Sigma$. Let $\rho$ be the induced $S$-automorphism. Then $\rho$ is either one-sided or balanced.
 Since $\Delta$ is factor-loopless, Lemma \ref{le:lelp} shows that $\rho$ cannot be balanced. Hence $\rho$ must be one-sided. By Lemma \ref{pr:prro}, the graph $\Delta$ admits a non-diagonal TF-automorphism. This contradicts the stability of $\Delta$. Thus no such $\Sigma$ exists, and $\Delta$ is Sym-prime.
\end{proof}

\section{The proof of Theorems~\ref{th:thgpsg} and~\ref{th:thfless}}\label{sec:secthppsp}

In this section, we give the proof of the main result of this paper.

\begin{lemma}\label{le:lecspsp}
Let $\Gamma_1$ and $\Gamma_2$ be two connected twin-free non-bipartite graphs.
Suppose that $\Gamma_1$ and $\Gamma_2$ are coprime with respect to the direct product. If $\Gamma_1$ and $\Gamma_2$ are both Sym-prime, then $\Gamma_1\times\Gamma_2$ is Sym-prime.
\end{lemma}

\begin{proof}

Suppose that $\Gamma_1\times\Gamma_2$ is not Sym-prime. Then there exists a bipartite graph $\Sigma$ such that the nontrivial graph pair $(\Gamma_1\times\Gamma_2,\Sigma)$ is a nontrivial unstable graph pair.
Since $(\Gamma_1\times\Gamma_2,\Sigma)$ is nontrivial, both $(\Gamma_1,\Sigma)$ and $(\Gamma_2,\Sigma)$ are nontrivial graph pairs.
Indeed, connectedness and twin-freeness are inherited by direct products in this setting, and any common direct-product factor would also be a common factor of $\Gamma_1\times\Gamma_2$ and $\Sigma$.
Since $\Gamma_2$ is Sym-prime, the pair $(\Gamma_2,\Sigma)$ is stable. Hence \[\aut(\Gamma_2\times\Sigma)=\aut(\Gamma_2)\times\aut(\Sigma).\]
Since $\Gamma_1$ is Sym-prime and $\Gamma_2\times\Sigma$ is bipartite, the pair $(\Gamma_1,\Gamma_2\times\Sigma)$ is stable. Hence \[\aut(\Gamma_1\times\Gamma_2\times\Sigma)=\aut(\Gamma_1)\times\aut(\Gamma_2\times\Sigma).\]
Combining these two equalities gives \[\aut(\Gamma_1\times\Gamma_2\times\Sigma)=\aut(\Gamma_1)\times\aut(\Gamma_2)\times\aut(\Sigma).\]
Since $\Gamma_1$ and $\Gamma_2$ are coprime connected non-bipartite twin-free graphs,
\[\aut(\Gamma_1\times\Gamma_2)=\aut(\Gamma_1)\times\aut(\Gamma_2).\] Therefore \[\aut(\Gamma_1\times\Gamma_2\times\Sigma)=\aut(\Gamma_1\times\Gamma_2)\times\aut(\Sigma),\]
which contradicts the instability of $(\Gamma_1\times\Gamma_2,\Sigma)$.
Thus $\Gamma_1\times\Gamma_2$ is Sym-prime.
\end{proof}

\begin{lemma}\label{le:lespisf}
Let $\Gamma=\Gamma_1\times\Gamma_2\times\cdots\times\Gamma_m$ be a direct product of connected non-bipartite graphs.
If $\Gamma$ is stable, then each $\Gamma_i$ is stable.
\end{lemma}

\begin{proof}
Suppose that some $\Gamma_i$ is unstable. Since $\Gamma_i$ is connected and non-bipartite, Lemma \ref{lem:TF} gives a non-diagonal TF-automorphism $(\alpha,\beta)\in\TFA(\Gamma_i)$.
We can extend the TF-automorphism to a TF-automorphism of $\Gamma$ by acting as $(\alpha,\beta)$ on the $i$-th coordinate and $\id$ on all other coordinates.
This gives a non-diagonal TF-automorphism of $\Gamma$. By Lemma \ref{lem:TF}, $\Gamma$ is unstable, contradicting the assumption. Hence every $\Gamma_i$ is stable.
\end{proof}

We now prove~Theorem \ref{th:thfless}.

\begin{proof}[Proof of Theorem~\ref{th:thfless}]

Suppose $$\Gamma=\Gamma_1^{a_1}\times\cdots\times\Gamma_m^{a_m}$$ is the prime direct product decomposition of $\Gamma$.
Since $\Gamma$ is stable, by Lemma \ref{le:lespisf}, each factor $\Gamma_i^{a_i}$ is stable.
By Theorem~\ref{th:thpsp}, $\Gamma_i^{a_i}$ is Sym-prime for each $1\leq i\leq m$. Combining this with Lemma \ref{le:lecspsp}, we obtain that  $\Gamma$ is Sym-prime.
\end{proof}

Combining the preceding results, we obtain ~Theorem \ref{th:thgpsg}.

\begin{proof}[Proof of Theorem~\ref{th:thgpsg}]
It was proved in \cite[Theorem~1.4]{WQX2025} that if $(\Gamma,\Sigma)$ is stable, then $\Gamma$ is stable. We only need to prove the ``if'' part.
Let $$\Gamma=\Gamma_1^{a_1}\times\cdots\times\Gamma_m^{a_m}$$ be the prime direct product decomposition of $\Gamma$.
Since $\Gamma$ is stable, each factor $\Gamma_i^{a_i}$, $1\leq i\leq m$, is stable. Therefore, by Theorem \ref{th:thfless}, $\Gamma$ is Sym-prime. Thus each nontrivial graph pair $(\Gamma,\Sigma)$ is stable. That is, if $\Gamma$ is stable, then for any bipartite graph $\Sigma$, the nontrivial graph pair $(\Gamma,\Sigma)$ is stable.
This establishes the ``if'' part.
\end{proof}

\section{Examples and counterexamples}\label{sec:seca}
In this section, loops are allowed. Let $\Gamma$ be a graph. Denote by $\ell(\Gamma)$ the number of vertices of $\Gamma$ that carry loops. We will construct explicit graph pairs that illustrate the necessity of the loopless condition in Conjecture \ref{conj:conjsps} and answer three questions posed by Gan, Liu and Xia \cite{GLX2025}.

\begin{definition}\label{de:detsc}
Let $T$ be a finite set, and let $R\subseteq T\times T$
be a symmetric binary relation. Define $\Gamma_R$ to be the graph with vertex set $V(\Gamma_R)=T\times T$. Two vertices $(a,b),(c,d)\in T\times T$ are adjacent if and only if $(a,d)\in R \text{ and } (c,b)\in R$. In addition, we define a bipartite graph $\Sigma_R$ with bipartition $U=\{x_i:i\in T\}$ and $W=\{y_j:j\in T\}$,
where $x_i$ is adjacent to $y_j$ if and only if $(i,j)\in R$.
\end{definition}

We now construct an automorphism of $\Gamma_R\times\Sigma_R$ which is a $\Sigma_R$-mixer.

\begin{lemma}\label{le:letsa}
Let $R\subseteq T\times T$ be a symmetric binary relation. Define a map $\phi_R:V(\Gamma_R\times\Sigma_R)\to V(\Gamma_R\times\Sigma_R)$ such that for any $((a,b),x_i)\in V(\Gamma_R)\times U$,
\[\bigl((a,b),x_i\bigr)^{\phi_R}
=
\bigl((a,i),x_b\bigr)
\]
and for any $((a,b),y_j)\in V(\Gamma_R)\times W$,
\[
\bigl((a,b),y_j\bigr)^{\phi_R}
=
\bigl((j,b),y_a\bigr).
\]
Then $\phi_R$ is an automorphism of $\Gamma_R\times\Sigma_R$ and $
\phi_R^2=\id$.
\end{lemma}

\begin{proof}
It is immediate from the definition that applying $\phi_R$ twice gives the identity.
Thus $\phi_R$ is a bijection.

It remains to check that $\phi_R$ maps edges of $\Gamma_R\times\Sigma_R$ to edges of $\Gamma_R\times\Sigma_R$. Since $\Sigma_R$ is bipartite, every edge of
$\Gamma_R\times\Sigma_R$ has one endpoint over the $U$-part and one endpoint over the
$W$-part. Take two vertices $\bigl((a,b),x_i\bigr)$ and $\bigl((e,f),y_j\bigr)$ in $V(\Gamma_R\times\Sigma_R)$.
The vertices $\bigl((a,b),x_i\bigr)$ and $\bigl((e,f),y_j\bigr)$ are adjacent in $\Gamma_R\times\Sigma_R$ if and only if $(a,b)$ is adjacent to $(e,f)$ in $\Gamma_R$ and $x_i$ is adjacent to $y_j$ in $\Sigma_R$.
By the definitions of $\Gamma_R$ and $\Sigma_R$, this is equivalent to $(a,f)\in R$, $(e,b)\in R$ and $(i,j)\in R$.

 Since $$\bigl((a,b),x_i\bigr)^{\phi_R}=\bigl((a,i),x_b\bigr),\qquad \bigl((e,f),y_j\bigr)^{\phi_R}=\bigl((j,f),y_e\bigr),$$ the vertices $(a,i)$ and $(j,f)$ are adjacent in $\Gamma_R$, and $x_b$ is adjacent to $y_e$ in $\Sigma_R$.
By Definition \ref{de:detsc}, we have $$(a,f), (j,i), (b,e)\in R.$$
Since $R$ is a symmetric relation, the condition $(j,i), (b,e)\in R$ is equivalent to $(i,j), (e,b)\in R$.
Thus $\bigl((a,b),x_i\bigr)^{\phi_R}$ is adjacent to $\bigl((e,f),y_j\bigr)^{\phi_R}$ in $\Gamma_R\times\Sigma_R$. Together with $\phi_R^2=\id$, we conclude that $\phi_R$ is an automorphism of $\Gamma_R\times\Sigma_R$.
\end{proof}

We now show an example from Definition \ref{de:detsc}, which will provide an answer to \cite[Question~3.5]{GLX2025}.

Take the relation $R=\{(0,1),(1,0),(1,1)\}$ in Definition \ref{de:detsc}. Let $\Delta$ be the graph (see Fig. \ref{fig:fig}) with vertex set $V(\Delta)=\{0,1\}\times\{0,1\}$ such that $(u, v)$ is adjacent to $(x,y)$ if and only if $(u,y)\in R$ and $(x,v)\in R$. Let $\Omega$ be the bipartite graph with vertex set $V(\Omega)=\{x_0,x_1, y_0,y_1\}$ and the bipartition $V(\Omega)=\{X, Y\}$ such that $X=\{x_0,x_1\}$ and $Y=\{y_0,y_1\}$ and $x_i$ is adjacent to $y_j$ if and only if $(i,j)\in R$. Clearly, $$E(\Omega)=\{\{x_0,y_1\}, \{x_1,y_0\}, \{x_1,y_1\}\}.$$
\begin{center}
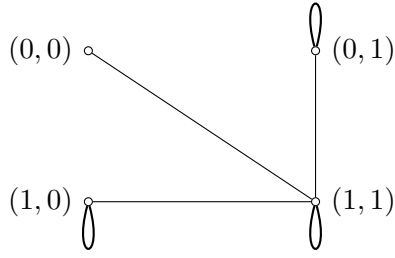
\begin{figure}[htbp]
\begin{tikzpicture}[
  vertex/.style={circle, fill=white, draw, minimum size=3pt, inner sep=0pt},
  loop/.style={looseness=65, thick}
]
\node[vertex] (v00) at (0,2) {}; 
\node[vertex] (v01) at (3,2) {}; 
\node[vertex] (v10) at (0,0) {}; 
\node[vertex] (v11) at (3,0) {}; 
\node[left=2pt] at (v00) {$(0,0)$};
\node[right=2pt] at (v01) {$(0,1)$};
\node[left=2pt] at (v10) {$(1,0)$};
\node[right=2pt] at (v11) {$(1,1)$};
\draw (v00) -- (v11);  
\draw (v01) -- (v11);  
\draw (v10) -- (v11);  
\draw (v01) edge[loop above, min distance=10pt] (v01);
\draw (v10) edge[loop below, min distance=10pt] (v10);
\draw (v11) edge[loop below, min distance=10pt] (v11);
%
\end{tikzpicture}
\caption{The graph $\Gamma_R$.}
\label{fig:fig}
\end{figure}
\end{center}
Note that $$V(\Delta)=  \{(0,0),(0, 1), (1,0),(1,1)\}$$ and
\begin{align*}
 N_\Delta((0,0))=&\ \{(1,1)\}, \\
 N_\Delta((0,1))=&\ \{(0,1),(1,1)\},\\
N_\Delta((1,0))=&\ \{(1,0),(1,1)\}, \\
N_\Delta((1,1))=&\ \{(0,0),(0,1),(1,0),(1,1)\},
\end{align*}
so $\Delta$ is connected, non-bipartite, and twin-free.
Next we prove that $\Delta$ is stable and prime with respect to the direct product.

\begin{lemma}\label{le:lecps}
The graph $\Delta$ is stable and prime with respect to the direct product.
\end{lemma}

\begin{proof}
We first prove that $\Delta$ is prime. Suppose that there are two nontrivial graphs $\Delta_1$ and $\Delta_2$ such that $\Delta\cong \Delta_1\times \Delta_2$.
Since $|V(\Delta)|=4$, both $\Delta_1$ and $\Delta_2$ must have order $2$. Note that a vertex $(x,y)$ has a loop in $\Delta_1\times \Delta_2$ if and only if $x$ has a loop in $\Delta_1$ and $y$ has a loop in $\Delta_2$. Hence
\[
\ell(\Delta_1\times\Delta_2)=\ell(\Delta_1)\ell(\Delta_2).
\]
Note that $\ell(\Delta)=3$, so $\ell(\Delta_1)\ell(\Delta_2)=3$. However, this is impossible because $\ell(\Delta_1),\ell(\Delta_2)\in\{0,1,2\}$. Therefore $\Delta$ is prime.

It remains to prove that $\Delta$ is stable. Suppose that it is unstable. Since $\Delta$ is non-bipartite, by \cite[Lemma~3.2(c)]{WQX2025}, there exists a non-diagonal TF-automorphism $(\alpha, \beta)\in \TFA(\Delta)$. Since $\alpha$ and $\beta$ are permutations of $V(\Delta)$, for each vertex $x\in V(\Delta)$, the size of $N_{\Delta}(x)$ is preserved under $\alpha$.
Since $$|N_\Delta((0,0))|=1,\qquad |N_\Delta((1,1))|=4,$$ we have \[(0,0)^\alpha=(0,0)^\beta=(0,0),\qquad
(1,1)^\alpha=(1,1)^\beta=(1,1).
\]
Hence both $\alpha$ and $\beta$ either fix $(0,1)$ and $(1,0)$, or interchange them.

We claim that $\alpha=\beta$. Suppose one of $\alpha,\beta$ interchanges $(0,1)$ and $(1,0)$, while the other fixes them.
Without loss of generality, we may assume that $\alpha$ interchanges $(0,1)$ and $(1,0)$ but $\beta$ fixes them.
Since $(0,1)$ has a loop, we have
$$((0,1),(0,1))\in A(\Delta).$$ Then
$$((0,1)^\alpha,(0,1)^\beta)=((1,0),(0,1))\in A(\Delta),$$
which is a contradiction.
Similarly, we can prove the other case. Hence $\alpha=\beta$.
Therefore all TF-automorphisms of $\Delta$ are diagonal. Combining this with Lemma \ref{lem:TF}, we conclude that $\Delta$ is stable.
\end{proof}

Next, we prove that $(\Delta,\Omega)$ is nontrivially unstable.

\begin{lemma}\label{le:legsnu}
The graph pair $(\Delta,\Omega)$ is nontrivially unstable.
\end{lemma}

\begin{proof}
We first prove that $(\Delta,\Omega)$ is a nontrivial graph pair. The graph $\Omega\cong P_4$ is connected and twin-free.
It remains to show that $\Delta$ and $\Omega$ are coprime with respect to the direct product. Since $\Omega\cong P_4$ is prime, loopless and $\Delta$ admits loops and is prime by Lemma \ref{le:lecps}, the result follows. Hence $(\Delta, \Omega)$ is a nontrivial graph pair.
From Lemma \ref{le:letsa}, we can construct an automorphism $\phi_R$ of $\Delta\times\Omega$ such that $\phi_R^2=\id$ and $\phi_R$ is an $\Omega$-mixer. Thus $\phi_R$ is an unexpected automorphism of $\Delta\times\Omega$, and so $(\Delta,\Omega)$ is unstable.
\end{proof}

\begin{remark}
 Lemma \ref{le:legsnu} shows that the condition ``$\Gamma$ is loopless'' is necessary in Conjecture \ref{conj:conjsps}.
As mentioned in Lemma \ref{le:lecps}, the graph $\Delta$ is prime and stable. However, by Lemma \ref{le:legsnu}, the pair $(\Delta,\Omega)$ is nontrivially unstable. Hence $\Delta$ is not Sym-prime.
Thus Conjecture \ref{conj:conjsps} is false when loops are allowed.
\end{remark}
From Lemma \ref{le:legsnu} we obtain answers to \cite[Question~3.5]{GLX2025}.

\begin{question}\cite[Question~3.5]{GLX2025}
\label{qu:qug1}
Is there a nontrivially unstable graph pair $(\Gamma,\Sigma)$, where $\Gamma$ is non-bipartite and $\Sigma$ is bipartite, such that there is no non-diagonal $\Sigma$-automorphism of $\Gamma$?
\end{question}

\begin{Answer}~\ref{qu:qug1}.
Let $\Gamma:=\Delta$ and $\Sigma:=\Omega$ be as in Lemma \ref{le:legsnu}.
 Lemma \ref{le:legsnu} implies that $(\Gamma,\Sigma)$ is unstable. However, there is no non-diagonal $\Sigma$-automorphism of $\Gamma$ since every non-diagonal $\Sigma$-automorphism would induce a non-diagonal TF-automorphism of $\Gamma$.
Note that $\Gamma$ admits a non-diagonal TF-automorphism if and only if $\Gamma$ is unstable by Lemma \ref{lem:TF}.
However, Lemma \ref{le:lecps} shows that $\Gamma$ is stable.
Thus $\Gamma$ does not admit non-diagonal TF-automorphisms, and consequently $\Gamma$ has no non-diagonal $\Sigma$-automorphisms.
 Thus there is a nontrivially unstable graph pair $(\Gamma,\Sigma)$ with the required properties in Question \ref{qu:qug1}.
\end{Answer}

We now recall Question~\ref{qu:qug2} and give a positive example.

\begin{question}\cite[Question~3.6]{GLX2025}
\label{qu:qug2}
Is there a nontrivially unstable graph pair $(\Gamma,\Sigma)$, where $\Gamma$ is non-bipartite and $\Sigma$ is bipartite, such that there exists a $\Sigma$-mixer of $\Gamma\times\Sigma$ and each connected component of $S(\Sigma)$ is Cartesian-quasicoprime to $S(\Gamma)$?
\end{question}

\begin{definition}\label{de:deg}
Let $A=\{a_0,a_1,a_2\}$ and $B=\{b_0,b_1\}$. Let $S\subseteq A\times B$ be the relation
\[
S=\{(a_0,b_0),(a_1,b_0),(a_1,b_1),(a_2,b_1)\}.
\]
Define a graph $\Gamma_S$ (see Fig. \ref{fig:fig1}) with vertex set $V(\Gamma_S)=A\times B$. Two vertices $(a,b),(c,d)\in V(\Gamma_S)$ are adjacent if and only if $(a,d)\in S$ and $(c,b)\in S$.
\end{definition}

\begin{center}
\begin{figure}[htbp]
\begin{tikzpicture}[
  vertex/.style={circle, fill=white, draw, minimum size=3pt, inner sep=0pt},
  loop/.style={looseness=65, thick}
]
\node[vertex] (v0) at (0,4) {};
\node[vertex] (v1) at (4,4) {};
\node[vertex] (v2) at (0,2) {};
\node[vertex] (v3) at (4,2) {};
\node[vertex] (v4) at (0,0) {};
\node[vertex] (v5) at (4,0) {};
\node[right=2pt] at (v0) {$(a_0,b_0)$};
\node[right=2pt] at (v1) {$(a_0,b_1)$};
\node[below=2pt] at (v2) {$(a_1,b_0)$};
\node[above=2pt] at (v3) {$(a_1,b_1)$};
\node[below=2pt] at (v4) {$(a_2,b_0)$};
\node[below=2pt] at (v5) {$(a_2,b_1)$};
\draw (v0) edge[loop left, min distance=12pt] (v0);
\draw (v2) edge[loop left, min distance=12pt] (v2);
\draw (v3) edge[loop right, min distance=12pt] (v3);
\draw (v5) edge[loop right, min distance=12pt] (v5);
\draw (v0) -- (v2);
\draw (v1) -- (v2);
\draw (v1) -- (v4);
\draw (v2) -- (v3);
\draw (v3) -- (v4);
\draw (v3) -- (v5);
\end{tikzpicture}
\caption{The graph $\Gamma_S$.}
\label{fig:fig1}
\end{figure}
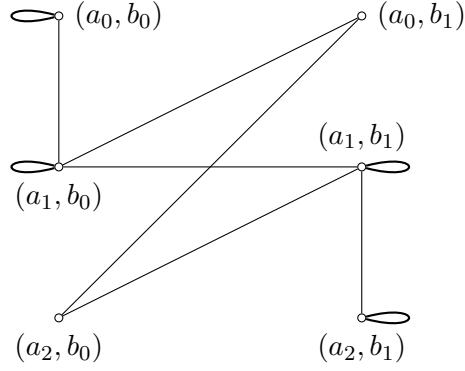
\end{center}

\begin{definition}\label{de:des}
Let $\Sigma_S$ be the bipartite graph with bipartition $U=\{x_{b_0},x_{b_1}\}$ and $W=\{y_{a_0},y_{a_1},y_{a_2}\}$,
where $x_j$ is adjacent to $y_i$ in $\Sigma_S$ if and only if $(i,j)\in S$. Clearly, $\Sigma_S\cong P_5$.
\end{definition}

From Fig.~\ref{fig:fig1}, it is easy to see that $\Gamma_S$ is connected.
 Moreover, we have
 \begin{align*}
 N_{\Gamma_S}((a_0,b_0))&=\{(a_0,b_0),(a_1,b_0)\},\\ N_{\Gamma_S}((a_0,b_1))&=\{(a_1,b_0),(a_2,b_0)\},\\
N_{\Gamma_S}((a_1,b_0))&=\{(a_0,b_0),(a_0,b_1),(a_1,b_0),(a_1,b_1)\},\\
N_{\Gamma_S}((a_1,b_1))&=\{(a_1,b_0),(a_1,b_1),(a_2,b_0),(a_2,b_1)\},\\
N_{\Gamma_S}((a_2,b_0))&=\{(a_0,b_1),(a_1,b_1)\},\\
N_{\Gamma_S}((a_2,b_1))&=\{(a_1,b_1),(a_2,b_1)\}.
\end{align*}
Clearly, $\Gamma_S$ is twin-free.

\begin{lemma}\label{le:leglx3.6}
The graph pair $(\Gamma_S,\Sigma_S)$ is a nontrivial graph pair and $\Gamma_S\times\Sigma_S$
admits a $\Sigma_S$-mixer.
\end{lemma}

\begin{proof}
Any common direct-product factor of $\Gamma_S$ and $\Sigma_S$ would have order greater than $1$ and would divide both $|V(\Gamma_S)|=6$ and $|V(\Sigma_S)|=5$. Since $\gcd(6,5)=1$, no such factor exists. Hence $\Gamma_S$ and $\Sigma_S$ are coprime with respect to the direct product. Note that $\Gamma_S$ and $\Sigma_S$ are connected, twin-free and exactly one of them, namely $\Sigma_S$, is bipartite. Hence $(\Gamma_S,\Sigma_S)$ is a nontrivial graph pair.

We define a map $\phi_S$ from $V(\Gamma_S\times\Sigma_S)$ to itself as follows:
 for vertices $((a,b),x_i)$ in $\Gamma_S\times U$, \[
\bigl((a,b),x_i\bigr)^{\phi_S}=\bigl((a,i),x_b\bigr),
\]
and for vertices $((a,b),y_j)$ in $\Gamma_S\times W$,
\[
\bigl((a,b),y_j\bigr)^{\phi_S}=\bigl((j,b),y_a\bigr).
\]
We now check that $\phi_S$ is an automorphism of $\Gamma_S\times\Sigma_S$. Consider two vertices
$\bigl((a,b),x_i\bigr)$ and $\bigl((e,f),y_j\bigr)$ in $V(\Gamma_S\times\Sigma_S)$.
They are adjacent in $\Gamma_S\times\Sigma_S$ if and only if $(a,b)$ is adjacent to $(e,f)$ in $\Gamma_S$ and $x_i$ is adjacent to $y_j$ in $\Sigma_S$.
By Definitions \ref{de:deg} and \ref{de:des}, this is equivalent to
\[
(a,f)\in S,\qquad
(e,b)\in S,\qquad
(j,i)\in S.
\]
The images are $\bigl((a,i),x_b\bigr)$ and $\bigl((j,f),y_e\bigr)$. These two image vertices are adjacent in $\Gamma_S\times\Sigma_S$ if and only if $(a,i)$ is adjacent to $(j,f)$ in $\Gamma_S$ and $x_b$ is adjacent to $y_e$ in $\Sigma_S$.
Again by Definitions \ref{de:deg} and \ref{de:des}, this is equivalent to
\[
(a,f)\in S,\qquad
(j,i)\in S,\qquad
(e,b)\in S.
\]
 Therefore $\phi_S$ preserves adjacency of $\Gamma_S\times\Sigma_S$. Combining this with $\phi_S^2=\id$, we have $\phi_S\in\aut(\Gamma_S\times\Sigma_S)$. Clearly, $\phi_S$ is a $\Sigma_S$-mixer.
\end{proof}

We now compute the Cartesian skeletons of $\Gamma_S$ and $\Sigma_S$.

\begin{lemma}\label{le:lecsg}
The Cartesian skeleton of $\Gamma_S$ is $S(\Gamma_S)\cong P_3\Box P_2$ and $S(\Sigma_S)=P_2\sqcup P_3$.
\end{lemma}

\begin{proof}
It is easy to check that $S(\Sigma_S)=P_2\sqcup P_3$. The Boolean square $B(\Gamma_S)$ has edges:
\begin{align*}
&\{(a_0,b_0),(a_1,b_0)\}, \qquad \{(a_0,b_1),(a_1,b_1)\}, \qquad \{(a_1,b_0),(a_2,b_0)\}, \\
&\{(a_1,b_1),(a_2,b_1)\}, \qquad \{(a_0,b_0),(a_0,b_1)\}, \qquad \{(a_1,b_0),(a_1,b_1)\}, \\
&\{(a_2,b_0),(a_2,b_1)\}, \qquad \{(a_0,b_0),(a_1,b_1)\}, \qquad \{(a_0,b_1),(a_1,b_0)\}, \\
&\{(a_1,b_0),(a_2,b_1)\}, \qquad \{(a_1,b_1),(a_2,b_0)\}.
\end{align*}
Note that
\begin{align*}
&N_{\Gamma_S}((a_0,b_0))\cap N_{\Gamma_S}((a_1,b_1))=\emptyset,\\
&N_{\Gamma_S}((a_0,b_0))\cap N_{\Gamma_S}((a_1,b_0))=\{(a_0,b_0),(a_1,b_0)\},\\
&N_{\Gamma_S}((a_1,b_1))\cap N_{\Gamma_S}((a_1,b_0))=\{(a_1,b_0),(a_1,b_1)\}.
\end{align*}
We have
\[N_{\Gamma_S}((a_0,b_0))\cap N_{\Gamma_S}((a_1,b_1))\subsetneq N_{\Gamma_S}((a_0,b_0))\cap N_{\Gamma_S}((a_1,b_0))\] and
\[N_{\Gamma_S}((a_0,b_0))\cap N_{\Gamma_S}((a_1,b_1))\subsetneq N_{\Gamma_S}((a_1,b_1))\cap N_{\Gamma_S}((a_1,b_0)).\]
So the edge $$\{(a_0,b_0),(a_1,b_1)\}$$ is dispensable.
Similarly, we obtain that $$\{(a_0,b_1),(a_1,b_0)\}, \qquad \{(a_1,b_0),(a_2,b_1)\}, \qquad \{(a_1,b_1),(a_2,b_0)\}$$ are dispensable. So the non-dispensable edges are
\begin{align*}
&\{(a_0,b_0),(a_1,b_0)\}, \qquad \{(a_0,b_1),(a_1,b_1)\},\qquad  \{(a_1,b_0),(a_2,b_0)\},\\
&\{(a_1,b_1),(a_2,b_1)\}, \qquad \{(a_0,b_0),(a_0,b_1)\}, \qquad \{(a_1,b_0),(a_1,b_1)\}, \qquad \{(a_2,b_0),(a_2,b_1)\}.
\end{align*}
These are exactly the edges of the Cartesian product $P_3\Box  P_2$.
Thus the Cartesian skeleton of $\Gamma_S$ is exactly $P_3\Box  P_2$.
\end{proof}

\begin{Answer}~\ref{qu:qug2}. Let $\Gamma:=\Gamma_S$ and $\Sigma:=\Sigma_S$ be defined as in Definitions \ref{de:deg} and \ref{de:des}. We know from Lemma \ref{le:lecsg} that $\Gamma$ and $\Sigma$ are Cartesian-quasicoprime.
By Lemma \ref{le:leglx3.6}, the graph pair $(\Gamma,\Sigma)$ is a nontrivial graph pair.
Moreover, we constructed a $\Sigma$-mixer of $\Gamma\times \Sigma$ in Lemma \ref{le:leglx3.6}. Therefore there exists a nontrivially unstable graph pair $(\Gamma,\Sigma)$, where $\Gamma$ is non-bipartite and $\Sigma$ is bipartite, such that there exists a $\Sigma$-mixer of $\Gamma\times\Sigma$ and each connected component of $S(\Sigma)$ is Cartesian-quasicoprime to $S(\Gamma)$.
\end{Answer}

At the end of this section, we give the answer to \cite[Question~3.7]{GLX2025}. We now restate the question as follows.

\begin{question}\cite[Question~3.7]{GLX2025}
\label{qu:qug3}
Is there a nontrivially unstable graph pair $(\Gamma,\Sigma)$, where $\Gamma$ is non-bipartite and $\Sigma$ is bipartite, such that there is a $\Sigma$-mixer of $\Gamma\times\Sigma$ and $\Gamma$ is factor-loopless?
\end{question}

In order to solve this question, we give the following key result.

\begin{lemma}\label{le:leosmi}
Let $(\Gamma,\Sigma)$ be a nontrivial graph pair, where $\Gamma$ is non-bipartite and
$\Sigma$ is bipartite. Then
$
\Gamma\times\Sigma$ admits no one-sided $\Sigma$-mixer.
\end{lemma}

\begin{proof}
Suppose, for a contradiction, that $\Gamma\times\Sigma$ admits a one-sided $\Sigma$-mixer.
As in the proof of Lemma \ref{pr:prro}, we may assume that $\rho_{+}$ is a $U$-mixer and $\rho_{-}$ is side-preserving. By Lemma~\ref{le:leosd} we may assume $S(\Gamma)=M\Box R$.

Since $M$ is nontrivial, there exist two distinct bijections
$f,g:V(M)\to V(N)$.
As in the proof of Lemma \ref{pr:prro}, both $(\xi_f,\zeta)$ and $(\xi_g,\zeta)$
are two-fold automorphisms of $\Gamma$, and the second component $\zeta$ is the same. Then $$(\xi_g\xi_f^{-1},\id)\in\TFA(\Gamma).$$
 Thus, for all $u,v\in V(\Gamma)$, $$(u,v)\in A(\Gamma)\Longleftrightarrow ((u^{\xi_g\xi_f^{-1}}),v)\in A(\Gamma).$$ This implies $N_\Gamma(u)=N_\Gamma(u^{\xi_g\xi_f^{-1}})$ for every vertex $u\in V(\Gamma)$.

Since $\Gamma$ is twin-free, we have $u^{\xi_g\xi_f^{-1}}=u$
for every $u\in V(\Gamma)$, and hence
$\xi_g\xi_f^{-1}=\id$.
Thus $\xi_g=\xi_f$.
But by construction, we have $$(m,r)^{\xi_f}=(\mu(f(m)),\alpha(r)),\qquad (m,r)^{\xi_g}=(\mu(g(m)),\alpha(r)).$$
Since $f\ne g$ and $\mu$ is injective, we have
$\xi_f\ne\xi_g$, a contradiction.
Therefore no one-sided $\Sigma$-mixer can exist.
\end{proof}

\begin{Answer}~\ref{qu:qug3}.
Suppose, for a contradiction, that there exists a nontrivially unstable graph pair $(\Gamma,\Sigma)$, where $\Gamma$ is non-bipartite and factor-loopless and $\Sigma$ is bipartite, such that $\Gamma\times\Sigma$ admits a $\Sigma$-mixer.
Let $\rho$ be the induced $S$-automorphism.
By Lemma~\ref{le:leosmi}, $\rho$ cannot be one-sided; hence it must be balanced.
But Lemma~\ref{le:lelp} shows that no balanced $\Sigma$-mixer can occur when $\Gamma$ is factor-loopless. This contradiction proves that no such pair exists.
\end{Answer}

\section*{Acknowledgments}
Wang is supported by the National Natural Science Foundation of China (12401453). Gao is supported by the National Natural Science Foundation of China (12571019), the Natural Science Foundation of Gansu Province (25JRRA644) and Innovative Fundamental Research Group Project of Gansu Province (23JRRA684). The authors are grateful to Binzhou Xia for his feedback on the first draft of the paper.


\begin{thebibliography}{30}

\bibitem{Big01}
N. Biggs, Algebraic Graph Theory, Cambridge University Press, Cambridge, 2001.



\bibitem{FH2022}
B.~Fernandez and A.~Hujdurovi\'{c},
Canonical double covers of circulants,
\textit{J.~Combin.~Theory Ser.~B}, 154 (2022), 49--59.

\bibitem{GLX2025}
Y.~Gan, W.~Liu and B.~Xia,
Unexpected automorphisms in direct product graphs,
\textit{J. Combin. Theory Ser. B}, 171 (2025), 140--164.

\bibitem{HIK2011}
R.~Hammack, W.~Imrich and S.~Klav\v{z}ar,
\textit{Handbook of Product Graphs, 2nd ed.}, CRC Press, 2011.


\bibitem{Hun00}
T.~W. Hungerford, Algebra, Springer, New York, 2003.

\bibitem{HMM2021}
A.~Hujdurovi\'{c}, D.~Mitrovi\'{c} and D.~W.~Morris,
On automorphisms of the double cover of a circulant graph,
\textit{Electron. J.~Combin.}, 28 (2021), no. 2, P4.43.

\bibitem{LM2011}
J.~Lauri and R.~Mizzi,
Two-fold automorphisms of graphs,
\textit{Australas.~J.~Combin.}, 49 (2011) 165--176.


\bibitem{LMS2015}
J.~Lauri, R.~Mizzi and R.~Scapellato,
Unstable graphs: a fresh outlook via TF-automorphisms,
\textit{Ars Math.~Contemp.}, 8 (2015) 115--131.

\bibitem{MSZ1989}
D.~Maru\v{s}i\v{c}, R.~Scapellato and N.~Zagaglia Salvi,
A characterization of particular symmetric $(0,1)$ matrices,
\textit{Linear Algebra Appl.}, 119 (1989), 153--162.

\bibitem{MSZ1992}
D.~Maru\v{s}i\v{c}, R.~Scapellato and N.~Zagaglia Salvi,
Generalized Cayley graphs,
\textit{Discrete Math.}, 102 (1992), 279--285.


\bibitem{Morris2021}
D.~W.~Morris,
On automorphisms of direct products of Cayley graphs on abelian groups,
\textit{Electron. J.~Combin.}, 28 (2021), no.~3, P3.5.

\bibitem{NS1996}
R.~Nedela and M.~\v{S}koviera,
Regular embeddings of canonical double coverings of graphs,
\textit{J.~Combin.~Theory Ser.~B}, 67 (1996), 249--277.

\bibitem{QXZ2019}
Y-L.~Qin, B.~Xia and S.~Zhou,
Stability of circulant graphs,
\textit{J.~Combin.~Theory Ser.~B}, 136 (2019), 154--169.

\bibitem{QXZ2021}
Y-L.~Qin, B.~Xia and S.~Zhou,
Canonical double covers of generalized Petersen graphs, and double generalized Petersen graphs,
\textit{J.~Graph Theory}, 97 (2021), 70--81.

\bibitem{QXZ2024}
Y-L.~Qin, B.~Xia and S.~Zhou,
Stability of graph pairs involving vertex-transitive graphs,
\textit{Discrete Math.}, 347 (2024), no.~4, Paper No.~113856, 6 pp.

\bibitem{QXZZ2021}
Y.-L. Qin, B.~Xia, J.-X. Zhou and S.~Zhou,
Stability of graph pairs,
\textit{J.~Combin.~Theory Ser.~B}, 147 (2021) 71--95.

\bibitem{Surowski2001}
D.~Surowski,
Stability of arc-transitive graphs,
\textit{J.~Graph Theory}, 38 (2001), 95--110.

\bibitem{Surowski2003}
D.~Surowski,
Automorphism groups of certain unstable graphs,
\textit{Math.~Slovaca}, 53 (2003), 215--232.

\bibitem{WQX2025} X. Wang, Y.-L. Qin, B. Xia,
The existence of unexpected automorphisms in direct product graphs, \url{https://arxiv.org/abs/2509.26170}.

\bibitem{WXZ2025}
X.~Wang, S.-J.~Xu and S.~Zhou,
Stability of graph pairs involving cycles,
\textit{Graphs Combin.}, 41 (2025) no.~2, Paper No.~49, 22 pp.



\bibitem{Wilson2008}
S.~Wilson,
Unexpected symmetries in unstable graphs,
\textit{J.~Combin.~Theory Ser.~B}, 98 (2008), 359--383.

\end{thebibliography}
\end{document}